\documentclass[12pt]{amsart}
\usepackage{mathrsfs}

\usepackage{amssymb,amsfonts,amsthm,amsmath}
\usepackage{epsfig}
\usepackage[utf8]{inputenc}
\usepackage{mathrsfs}
\usepackage{lscape}
\usepackage{amssymb,amsmath,graphicx,color,textcomp, amsthm,bbm,bbold, enumerate,booktabs}
\usepackage{chngcntr}
\usepackage{apptools}
\AtAppendix{\counterwithin{theorem}{section}}
\definecolor{Red}{cmyk}{0,1,1,0}

\definecolor{verde}{cmyk}{1,0,1,0}

\definecolor{loka}{cmyk}{.5,0,1,.5}
\definecolor{azul}{cmyk}{1,1,0,0}

\numberwithin{equation}{section}

\newcommand{\be}{\begin{equation}}
\newcommand{\ee}{\end{equation}}

\newtheorem{definition}{Definition}

\newtheorem{lemma}{Lemma}
\newtheorem{teorema}{Theorem}

\begin{document}
\title{Existence of mild solutions to Hilfer fractional evolution equations in Banach space}
\author{J. Vanterler da C. Sousa}
\address{ Department of Applied Mathematics, Institute of Mathematics,
 Statistics and Scientific Computation, University of Campinas --
UNICAMP, rua S\'ergio Buarque de Holanda 651,
13083--859, Campinas SP, Brazil\newline
e-mail: {\itshape \texttt{ra160908@ime.unicamp.br}}}

\begin{abstract} In this paper, we investigate the existence of mild solutions to Hilfer fractional equation of semi-linear evolution with non-instantaneous impulses, using the concepts of equicontinuous $C_{0}$-semigroup and Kuratowski measure of non-compactness in Banach space $\Omega$.

\vskip.5cm
\noindent
\emph{Keywords}: Hilfer fractional evolution equations, non-instantaneous impulses, mild solution, existence, equicontinuous $C_{0}$-semigroup, Kuratowski measure of non-compactness.
\newline 
%MSC 2010 subject classifications. 26A33, 34A08, 34A12, 34K20, 37C25 .
\end{abstract}
\maketitle

%%%%%%%%%%%%%%%%%%%%%%%%%%%%%%%%%%%%%%%%%%%%%%%%%%%%%%%%%%%%%%%%%%%%%%%%%%%%%%%%%%%%%%%%%%%%%%%%%%%%%%%%%%%%%%%%%%%%%%%%%%%%%%%%%%%%%%%%%%%%%%%%%%%%%%%%%%%%%

\section{Introduction}

We can start the paper with the following questions: Why study fractional calculus? What are the advantages we gain by investigating and proposing new results in the field of fractional calculus? Are the results presented so far important and relevant to the point of contributing to the scientific community? In a simple and clear answer, it is enough to notice the exponential scientific growth in the area and the impact that the fractional calculus has contributed to mathematics and other diverse sciences, specially, in a shocking way in the context of mathematical modeling \cite{KSTJ,ECJT,der1,appl,apll1,der2,saba,mainardi}.

The theory of differential equations with non-instantaneous impulses and impulsive evolution equations in Banach spaces has been investigated by many researchers in the last decades \cite{lakshmikantham,benchohra,samoilenko}. It is note that investigating the existence, uniqueness, stability of solutions of differential equations of evolution, has been object of study and applicability in the scientific community, since it describes processes that experience a sudden change in their states at certain moments \cite{principal,med,haoa,mild,chen2017}. The applicability of the obtained results related to the differential equations, especially with non-instantaneous impulses, can be found in several areas, such as: physics, engineering, economics, biology, medicine and mathematics itself, among others \cite{lakshmikantham,benchohra}.

In 2013 Pierri et al. \cite{pierri}, investigated the existence of solutions of a class of abstract semi-linear differential equations with non-instantaneous impulses using the semigroup analytic theory. In the same year, Hernandez and O'Regan \cite{hernandez}, investigated the existence of solutions of a new class of impulsive abstract differential equations with non-instantaneous impulses. In this sense, Zhang and Li \cite{chen}, by means of monotone iterative technique and operator semigroup theorem, investigated the existence of mild solutions of a class semi-linear evolution equations, also with non-instantaneous impulses in the space of Banach. So it is noted that the study on the subject is indeed important and motivating for the researchers.

Therefore, many researchers, specifically, from the fractional analysis group, had the motivation, based on tools and new results from the fractional calculus, to investigate the essential properties of solutions of fractional differential equations, from existence, uniqueness, Ulam-Hyers stabilities, controllability, among others, thus providing an exponential growth of new results and strengthening the area \cite{est1,est2,est3,est4,est5,est6,attrac,contr,contr1,contr2}.

In 2015 Gu and Trujillo \cite{exist2}, were concerned to focus on the investigation of the existence of mild solutions of the evolution of fractional equation order in the sense of Hilfer fractional derivative, using noncompactness measure in the space of Banach $X$. On the other hand, Fu and Huang \cite{deca}, investigated the existence and regularity of solutions for a neutral functional integro-differential equation with state-dependent delay in Banach space, using the fixed-point theorem of Sadovskii under compactness condition for the resolvent operator and under Hölder continuity condition.

In order to propose new results that contribute to the field of fractional differential equations, Mu \cite{exist4} decided to investigate the existence of a mild solution for the impulsive fractional evolution equation given by
\begin{equation*}
\left\{ 
\begin{array}{rll}
D_{0+}^{\alpha}u\left( t\right)+ Au(t)& = &f\left(t,u\left( t\right) \right) ,\,\,\,\,\,\text{ }t\in I:=[0,T], t\neq t_{k}\\ 
\Delta u |_{t=t_{k}} & = & I_{k}(u(t_{k})) ,\,\,\,\,\,\text{ } k=1,2,...,m \\ 
u\left( 0\right)+g(u) & = & x_{0}\in X%
\end{array}%
\right.
\end{equation*}
where $D_{0+}^{\alpha }(\cdot)$ is Caputo fractional derivative with $0<\alpha<1$, $A: D(A)\subset X: \rightarrow X$ is a linear closed densely defined operator, –$A$ is the infinitesimal generator of an analytic semigroup of uniformly bounded linear operators $(T(t))_{t\geq 0}$, $0=t_{0}<t_{1}<t_{2}<\cdots <t_{m}<t_{m+1}=T$, $f:I \times  \rightarrow X$ is continuous, $g:PC(I,X)\rightarrow X$ is continuous, the impulsive function $I_{k}: X \rightarrow X$ is continuous, $\Delta u |_{t=t_{k}}=u(t^{+})-u(t^{-})$, where $u(t^{+})$, $u(t^{-})$ represent the right and left limits of $u(t)$ at $t = t_{k}$, respectively. For a brief reading on others of existence and uniqueness of solutions of fractional differential equations, we suggest \cite{exist,exist3,exist6,gou,est7}.

In this sense, inspired by the works proposed up to here, we note the necessity and importance of the realization of a work that will contribute to the existence of solutions of fractional differential equations. Then, we consider the semi-linear evolution fractional differential equation with non-instantaneous impulses in Banach space $\Omega$, given by
\begin{equation}\label{eq1}
\left\{ 
\begin{array}{rll}
^{H}\mathbb{D}_{0+}^{\alpha ,\beta }u\left( t\right) +\mathcal{A}\left( t\right) & = & f\left(
t,u\left( t\right) \right) ,\,\,\,\,\,\text{ }t\in \overset{m}{\underset{k=0}{\bigcup }%
}\left( s_{k},t_{k+1}\right] \\ 
u\left( t\right) & = & \zeta _{k}\left( t,u\left( t\right) \right) ,\,\,\,\,\,\text{ }%
t\in \overset{m}{\underset{k=1}{\bigcup }}\left( t_{k},s_{k}\right] \\ 
I_{0+}^{1-\gamma }u\left( 0\right) & = & u_{0}%
\end{array}%
\right.
\end{equation}
where $^{H}\mathbb{D}_{0+}^{\alpha ,\beta }\left( \cdot \right) $ is Hilfer fractional derivative, $I_{0+}^{1-\gamma }\left( \cdot \right) $ is Riemann-Liouville fractional integral with $0<\alpha \leq 1$, $0\leq \beta \leq 1$ and $0\leq \gamma \leq 1\left( \gamma =\alpha +\beta \left( 1-\alpha \right) \right)$, $\mathcal{A}:\mathfrak{D}\left( \mathcal{A}\right) \subset \Omega \rightarrow \Omega $ is a linear operator and is the infinitesimal generator of a strongly continuous semigroup $\left( C_{0}-\text{semigroup}\right) $ $\left( \mathbb{P}\left(t\right) \right) _{t\geq 0}$ in $\Omega $ with $0<t_{1}<t_{2}<\cdot \cdot \cdot <t_{m}<t_{m+1}:=a,$ $a>0$ is a constant, $s_{0}:=0$ and $s_{k}\in \left( t_{k},t_{k+1}\right) $ for each $k=1,2,...,m$. We also have $f:\left[ 0,a\right] \times \Omega \rightarrow \Omega $ a given nonlinear function satisfying some assumptions $\zeta_{k}:\left( t_{k},s_{k}\right] \times \Omega \rightarrow \Omega $ is non-instantaneous impulsive function for all $k=1,2,...m,$ and $u_{0}\in \Omega$.

We highlight here a rigorous analysis of Eq.(\ref{eq1}) regarding the main results and advantages obtained in this paper:

\begin{itemize}

\item From the limits $\beta \rightarrow 0$ and $\beta \rightarrow 1$ in Eq.(\ref{eq1}), we obtain the respective special cases for the differential equations, that is, the classical fractional derivatives of Riemann-Liouville and Caputo, respectively. In addition to the integer case, by choosing $\alpha = 1 $;

\item An important and relevant factor are the properties of the Hilfer fractional derivative, since they are preserved for their particular cases. In this sense, when investigating a particular property of a fractional differential equation and obtaining a particular case for the derivative, the properties of the differential equation are preserved, in this case, the existence of mild solutions;

\item We present a new class of solutions for differential equation of semi-linear evolution with non-instantaneous impulses by means of the Hilfer fractional derivative;

\item We use the concepts of $C_{0}$-semigroup equicontinuous and Kuratowski measure of noncompactness $\mu(\cdot)$, among others that in the preliminary section are presented, in particular, Lebesgue dominated convergence theorem, and investigate the existence of mild solutions of Eq.(\ref{eq1}) in the space $\mathbf{PC}_{1-\gamma}(J,\Omega)$;

\item Since many applications are performed by means of differential equations with non-instantaneous impulses, specifically in biology and medicine; and by the enormous amount of parameters that appear when using differential equations to model a given problem, one way to overcome a certain barrier is to propose more general differential equations. In this case, one way is to use the more general fractional derivatives and here, we use the Hilfer fractional derivative. Although there are other fractional derivatives, a special emphasis on the $\psi$-Hilfer fractional derivative, we restrict this work to the Hilfer fractional derivative. In this sense, the result obtained here may also contribute to future applications.

\end{itemize}

This paper is organized as follows. In section 2 we present the space of the weighted functions and their respective norm, as well as the concepts of Hilfer fractional derivative. The concepts of $C_{0}$-semigroup equicontinuous, Kuratowski measure of noncompactness $\mu(\cdot)$, mild solution and Lemmas results that are of paramount importance throughout the paper are also presented. In section 3, the main result of the paper is investigated, the existence of mild solution of the semi-linear evolution fractional differential equation with non-instantaneous impulses in the Banach $\Omega$ space, making use of refined mathematical analysis tools, in particular, of Lebesgue dominated convergence theorem. Concluding remarks close the paper.

%%%%%%%%%%%%%%%%%%%%%%%%%%%%%%%%%%%%%%%%%%%%%%%%%%%%%%%%%%%%%%%%%%%%%%%%%%%%%%%%%%%%%%%%%%%%%%%%%%%%%%%%%%%%%%%%%%%%%%%
%%%%%%%%%%%%%%%%%%%%%%%%%%%%%%%%%%%%%%%%%%%%%%%%%%%%%%%%%%%%%%%%%%%%%%%%%%%%%%%%%%%%%%%%%%%%%%%%%%%%%%%%%%%%%%%%%%%%%%%
%%%%%%%%%%%%%%%%%%%%%%%%%%%%%%%%%%%%%%%%%%%%%%%%%%%%%%%%%%%%%%%%%%%%%%%%%%%%%%%%%%%%%%%%%%%%%%%%%%%%%%%%%%%%%%%%%%%%%%%

\section{Preliminaries}

In this section, some definitions and results are presented through Lemmas, essential for the development of the paper.

The space of continuous function $C\left( J,\mathbb{R}\right) $ $(J:=[0,a])$ with norm is given by \cite{est1}
\begin{equation*}
\left\Vert u\right\Vert =\underset{t\in J}{\sup }\left\Vert u\left( t\right)
\right\Vert .
\end{equation*}

On the other hand, the weighted space of functions $u$ on $J^{\prime }:=\left( 0,a\right] $ is
defined by
\begin{equation*}
C_{1-\gamma }\left( J,\Omega \right) =\left\{ u\in C\left( J^{\prime },\Omega \right) ,\text{ }t^{1-\gamma }u\left( t\right) \in C\left( J,\Omega \right) \right\}
\end{equation*}
where $0\leq \gamma \leq 1,$ with the norm
\begin{equation*}
\left\Vert u\right\Vert _{C_{1-\gamma }}=\underset{t\in J^{\prime }}{\sup }\left\Vert t^{1-\gamma }u\left( t\right) \right\Vert
\end{equation*}
obviously, the space $C_{1-\gamma }\left( J,\Omega \right) $ is a Banach space. We now, present the definition piecewise space of functions $u$ on $\mathbf{PC}_{1-\gamma }\left( \left( t_{k},t_{k+1}\right] ,\Omega \right)$, given by \cite{est6}
\begin{equation*}
\mathbf{PC}_{1-\gamma }\left( J,\Omega \right) =\left\{ 
\begin{array}{c}
\left( t-t_{k}\right) ^{1-\gamma }u\left( t\right) \in C_{1-\gamma }\left(
\left( t_{k},t_{k+1}\right],\mathbb{R}\right) \\ 
\underset{t\rightarrow t_{k}}{\lim }\left( t-t_{k}\right) ^{1-\gamma
}u\left( t\right) ,\text{ and exists for }k=1,...,m
\end{array}
\right\}
\end{equation*}
where the norm is given by
\begin{equation*}
\left\Vert u\right\Vert _{\mathbf{PC}_{1-\gamma }}=\max \left\{ \underset{t\in J}{%
\sup }\left\Vert t^{1-\gamma }u\left( t^{+}\right) \right\Vert ,\underset{%
t\in J}{\sup }\left\Vert t^{1-\gamma }u\left( t^{-}\right) \right\Vert
\right\}
\end{equation*}
where $u\left( t^{+}\right) $ and $u\left( t^{-}\right) $ represent
respectively the right and left limits of $u\left( t\right) $ at $t\in J.$

Consider the following set
\begin{equation*}
\Lambda _{r}=\left\{ u\in \mathbf{PC}_{1-\gamma }\left( J,\Omega \right) ,\text{ }%
\left\Vert u\left( t\right) \right\Vert _{C_{1-\gamma }}<r,\text{ }t\in
J\right\}.
\end{equation*}
Note that, for each finite constant $r>0$, $\Lambda _{r}$ is bounded, closed and convex set in $\mathbf{PC}_{1-\gamma }\left( J,\Omega \right) $.

We denote by $\mathcal{L}\left( \Omega \right) $ be the Banach space for all linear and bounded
operators on $\Omega$. Since the semigroup $\left( \mathbb{P_{\alpha,\beta}}\left( t\right) \right)
_{t\geq 0}$ generalized by $(-\mathcal{A})$ is a $C_{0}$-semigroup in $\Omega $, denote
\begin{equation}\label{eq2}
\mathbf{M}:=\sup \left\Vert \mathbb{P}_{\alpha ,\beta }\left( t\right) \right\Vert _{\mathcal{L}\left(
\Omega \right) }
\end{equation}
then $\mathbf{M}\geq 1$ is a finite number.

On the other hand, let $n-1<\alpha \leq n$ with $n\in \mathbb{N}$ and $f,\psi \in C^{n}\left(J,\mathbb{R}\right) $ be two functions such that $\psi $ is increasing and $\psi ^{\prime }\left( t\right) \neq 0$, for all $t\in J.$ The left-sided $\psi$-Hilfer fractional derivative $^{H}\mathbb{D}_{0+}^{\alpha ,\beta }\left( \cdot \right) $ of a function $f$ of order $\alpha $ and type $0\leq \beta \leq 1$ is defined by \cite{der1}
\begin{equation*}
^{H}\mathbb{D}_{0+}^{\alpha ,\beta }u\left( t\right) =I_{0+}^{\beta \left( n-\alpha
\right) ;\psi }\left( \frac{1}{\psi ^{\prime }\left( t\right) }\frac{d}{dt}%
\right) ^{n}I_{0+}^{\left( 1-\beta \right) \left( n-\alpha \right) ;\psi
}u\left( t\right),
\end{equation*}
where $I^{\alpha}_{0+}(\cdot)$ is $\psi$-Riemann-Liouville fractional integral. The right-left  $\psi$-Hilfer fractional derivative $^{H}\mathbb{D}_{0-}^{\alpha ,\beta }\left( \cdot \right) $ is defined in an analogous way \cite{der1}. 

Choosing $\psi \left( t\right) =t$, we have the Hilfer fractional derivative, given by 
\begin{equation}\label{eq4}
^{H}\mathbb{D}_{0+}^{\alpha ,\beta }u\left( t\right) =I_{0+}^{\beta \left( n-\alpha
\right) }\left( \frac{d}{dt}\right) ^{n}I_{0+}^{\left( 1-\beta \right)
\left( n-\alpha \right) }u\left( t\right) .
\end{equation}

For the development of this paper, we use the Hilfer fractional derivative, Eq.(\ref{eq4}). The following are some fundamental concepts and results for obtaining the principal of this paper.

\begin{definition}{\rm \cite{principal,med}}\label{def1} A $C_{0}$-semigroup $\left( \mathbb{P}\left( t\right) \right) _{t\geq 0}$ in $\Omega $ is said to be equicontinuous if $\mathbb{P}\left( t\right) $ is continuous by operator for every $t>0$.
\end{definition}

\begin{definition}{\rm\cite{principal,med}}\label{def2} The Kuratowski measure of non-compactness $\mu \left( \cdot
\right) $ defined on bounded set $S$ of Banach space $\Omega $ is
\begin{equation*}
\mu \left( \mathbf{S}\right) :=\inf \left\{ \delta >0:\mathbf{S}=\overset{m}{\underset{i=1}{%
\bigcup }}\mathbf{S}_{i}with\,\, {\rm{diam}}\left( S_{i}\right) \leq \delta ,\text{ for }%
i=1,2,...,m\right\} .
\end{equation*}
\end{definition}

The following properties about the Kuratowski measure of non-compactness are
well known.

\begin{lemma}{\rm\cite{principal,med}}\label{lem1} Let $\Omega $ be a Banach space and $\mathbf{S},\mathbf{U}\subset \Omega $ be bounded.
The following properties are satisfied:
\begin{enumerate}
\item  $\mu \left( \mathbf{S}\right) =0$ if and only if $\overline{\mathbf{S}}$ is compact, where 
$\overline{\mathbf{S}}$ means the closure hull of $\mathbf{S}$;

\item $\mu \left( \mathbf{S}\right) =\mu \left( \overline{\mathbf{S}}\right) =\mu \left( conv\text{ }\mathbf{S}\right) $, where $conv$ $\mathbf{S}$ means the convex hull of $\mathbf{S}$;

\item $\mu \left( \lambda \mathbf{S}\right) =\left\vert \lambda \right\vert \mu \left(\mathbf{S}\right) $ for any $\lambda \in \mathbb{R};$

\item $\mathbf{S}\subset U$ implies $\mu \left( \mathbf{S}\right) \leq \mu \left( \mathbf{U}\right) ;$

\item $\mu \left( \mathbf{S}\cup \mathbf{U}\right) =\max \left\{ \mu \left( \mathbf{S}\right) ,\mu \left(
\mathbf{U}\right) \right\} ;$

\item $\mu \left( \mathbf{S}+\mathbf{U}\right) \leq \mu \left( \mathbf{S}\right) +\mu \left( \mathbf{U}\right) ,$
where $\mathbf{S}+\mathbf{U}=\left\{ x/x=y+z,\text{ }y\in \mathbf{S},\text{ }z\in \mathbf{U}\right\} ;$

\item If $\mathbf{Q}:\mathfrak{D}\left( \mathbf{Q}\right) \subset \Omega \rightarrow \Omega $ is Lipschitz continuous with constant $k,$ then $\mu \left( \mathbf{Q}\left( \mathbf{V}\right) \right) \leq k\mu \left( \mathbf{V}\right) $ for any bounded subset $\mathbf{V}\subset \mathfrak{D}\left( \mathbf{Q}\right) ,$ where $\Omega $ is another Banach space.
\end{enumerate}
\end{lemma}

Some notations are necessary in order to facilitate and for the better development of the results during the paper. Among these, we denote by $\mu \left( \cdot \right)$, $\mu _{C_{1-\gamma }}\left( \cdot \right) $ and $\mu _{\mathbf{PC}_{1-\gamma }}\left( \cdot \right) $ the Kuratowski measure of non-compactness, on the bounded set of $\Omega $, $C_{1-\gamma }\left( J,\Omega \right) $ and $\mathbf{PC}_{1-\gamma }\left( J,\Omega \right) $, respectively. 

An important Kuratowski property measure of non-compactness is as follows: for any $\mathfrak{D}\subset C_{1-\gamma }\left( J,\Omega \right) $ and $t\in J,$ set $\mathfrak{D}\left( t\right) =\left\{ u\left( t\right) /u\in \mathfrak{D}\right\} $ then $\mathfrak{D}\left( t\right) \subset \Omega.$ If $\mathfrak{D}\subset C_{1-\gamma }\left( J,\Omega \right) $ is bounded, then $\mathfrak{D}\left( t\right) $ is bounded in $\Omega $ and $\mu \left( \mathfrak{D}\left( t\right) \right) \leq \mu _{C_{1-\gamma }}\left( \mathfrak{D}\right)$. 

\begin{definition}{\rm\cite{principal,med}}\label{def3} Let $\Omega $ be a Banach space, and let $\mathbf{S}$ be a nonempty subset of $\Omega .$\ A continuous operator $\mathbf{Q}:\mathbf{S}\rightarrow \Omega $ is called to be $k-$set-contractive if there exists a constant $k\in \left[ 0,1\right) $ such that, for every bounded set $\widetilde{\Omega }\subset \mathbf{S}$ 
\begin{equation*}
\mu \left( \mathbf{Q}\left( \widetilde{\Omega }\right) \right) \leq k\mu \left( 
\widetilde{\Omega }\right) .
\end{equation*}
\end{definition}

\begin{lemma}{\rm\cite{principal,med}}\label{lem2} Assume that $\widetilde{\Omega }\subset \Omega $ is a bounded closed and convex set on $\Omega $ ($\Omega $ is Banach space), the operator $\mathbf{Q}:\widetilde{\Omega }\rightarrow \widetilde{\Omega }$ is $k$-set-contractive. Then $Q$ has at least one fixed point in $\widetilde{\Omega }$.
\end{lemma}

\begin{lemma}{\rm\cite{principal,med}}\label{lem3} Let $\Omega $ be a Banach space, and let $\mathfrak{D}\subset \Omega $ be bounded. Then there exits a countable set $\mathfrak{D}_{0}\subset \mathfrak{D}$, such that $\mu \left( \mathfrak{D}\right) \leq 2\mu \left( \mathfrak{D}_{0}\right)$.
\end{lemma}

\begin{lemma}{\rm\cite{principal,med}}\label{lem4} Let $\Omega $ be a Banach space, and let $\mathfrak{D}=\left\{ u_{n}\right\} \subset \mathbf{PC}_{1-\gamma }\left( \left[ b_{1},b_{2}\right] ,\Omega \right) $ be a bounded and countable set for constants $-\infty <b_{1}<b_{2}<\infty $. Then $\mu \left( \mathfrak{D}\left( t\right) \right) $ is Lebesgue integral on $\left[ b_{1},b_{2}\right] $ and
\begin{equation*}
\mu \left( \left\{ \int_{b_{1}}^{b_{2}}u_{n}\left( t\right) dt:n\in \mathbb{N} \right\} \right) \leq 2\int_{b_{1}}^{b_{2}}\mu \left( \mathfrak{D}\left( t\right) \right) dt.
\end{equation*}
\end{lemma}

\begin{lemma}{\rm\cite{principal,med}}\label{lem5} Let $\Omega $ be a Banach space, and let $\mathfrak{D}\subset C_{1-\gamma }\left( \left[ b_{1},b_{2}\right] ,\Omega \right) $ be bounded and equicontinuous. Then $\mu \left( \mathfrak{D}\left( t\right) \right) $ is continuous on $\left[ b_{1},b_{2}\right]$, and 
\begin{equation*}
\mu _{C_{1-\gamma }}\left( \mathfrak{D}\right) =\underset{t\in \left[ b_{1},b_{2}\right]}{\max }\mu \left( \mathfrak{D}\left( t\right) \right) .
\end{equation*}
\end{lemma}

\begin{definition}{\rm\cite{est1,exist2,exist}}\label{def4} Suppose that $\mathcal{A}$ is a closed, densely defined linear operator on $\Omega $. A family $\left\{ \mathbb{P}_{\alpha }\left( t\right) /t\geq 0\right\} \subset \mathbf{B}\left( \Omega \right) $ is called an $\alpha$-times resolvent family generator by $\mathcal{A}$ if the following conditions are satisfied:
\begin{enumerate}
\item $\mathbb{P}_{\alpha }\left( t\right) $ is strongly continuous on $\mathbb{R} _{+}$ and $\mathbb{P}_{\alpha }\left( 0\right) =I$;

\item $\mathbb{P}_{\alpha }\left( t\right) \mathfrak{D}\left( \mathcal{A}\right) \subset \mathfrak{D}\left( \mathcal{A}\right) $ and $\mathcal{A}\mathbb{P}_{\alpha }\left( t\right) x=\mathbb{P}_{\alpha }\mathcal{A}x$ for all $x\in \mathfrak{D}\left( \mathcal{A}\right) ,$ $t\geq 0$;

\item For all $x\in \mathfrak{D}\left( \mathcal{A}\right) $ and $t\geq 0$, $\mathbb{P}_{\alpha }\left( t\right) x=x+I_{0+}^{\alpha }\mathbb{P}_{\alpha }\left( t\right) \mathcal{A}x$.
\end{enumerate}
\end{definition}

\begin{lemma}{\rm\cite{est1,exist2,exist}}\label{lem6} The fractional nonlinear differential equation {\rm Eq.(\ref{eq1})} is equivalent to the integral equation
\begin{equation}\label{eq5}
u\left( t\right) =\left\{ 
\begin{array}{cl}
\dfrac{t^{\gamma -1}}{\Gamma \left( \gamma \right) }u_{0}+\dfrac{1}{\Gamma
\left( \alpha \right) }\displaystyle\int_{0}^{t}\left( t-s\right) ^{\alpha -1}\left(
f\left( s,u\left( s\right) \right) -\mathcal{A}u\left( s\right) \right) ds, & t\in 
\left[ 0,t_{1}\right]  \\ 
\xi _{i}\left( t,u\left( t\right) \right) , & t\in \left( t_{i},s_{i}\right]
,\text{ }i=1,...,m \\ 
\zeta_{k}\left( t,u\left( t\right) \right) +\dfrac{1}{\Gamma \left( \alpha
\right) }\displaystyle\int_{0}^{t}\left( t-s\right) ^{\alpha -1}\left( f\left( s,u\left(
s\right) \right) -\mathcal{A}u\left( s\right) \right) ds & t\in \left( s_{i},t_{i+1}
\right] ,\text{ }i=1,...,m%
\end{array}%
\right. .
\end{equation}
\end{lemma}

The following is the definition of the Wright function, fundamental in mild solution of the Eq.(\ref{eq1}). Then, the Wright function $\mathbf{M}_{\alpha }\left( \mathbf{Q}\right) $ is defined by
\begin{equation*}
\mathbf{M}_{\alpha }\left( \mathbf{Q}\right) =\overset{\infty }{\underset{n=1}{\sum }}\frac{\left( -\mathbf{Q}\right) ^{n-1}}{\left( n-1\right) !\Gamma \left( 1-\alpha n\right) } ,\text{ }0<\alpha <1,\text{ }\mathbf{Q}\in \mathbb{C}
\end{equation*}
satisfying the equation
\begin{equation*}
\int_{0}^{\infty }\theta ^{\overline{\delta }}\mathbf{M}_{\alpha }\left( \theta \right) d\theta =\frac{\Gamma \left( 1+\overline{\delta }\right) }{\Gamma \left( 1+\alpha \overline{\delta }\right) },\text{ for }\theta \geq 0.
\end{equation*}

\begin{definition}{\rm\cite{est1,exist2,exist}} A function $u \in \mathbf{PC}_{1-\gamma }\left( J,\Omega \right) $ is called a mild solution of {\rm Eq.(\ref{eq1})}, if the integral {\rm Eq.(\ref{eq5})} holds, we have
\begin{equation*}
u\left( t\right) =\left\{ 
\begin{array}{rl}
\mathbb{P}_{\alpha ,\beta }\left( t\right) u_{0}+\displaystyle\int_{0}^{t}\mathbf{K}_{\alpha }\left(t-s\right) f\left( s,u\left( s\right) \right) ds, & t\in \left[ 0,t_{1}\right]  \\ 
\xi _{i}\left( t,u\left( t\right) \right) , & t\in \left( t_{i},s_{i}\right]
,\,\text{ }i=1,...,m \\ 
\mathbb{P}_{\alpha ,\beta }\left( t\right) \zeta_{k}\left( s_{k},u\left( s_{k}\right) \right) +\displaystyle\int_{s_{i}}^{t}\mathbf{K}_{\alpha }\left( t-s\right) f\left( s,u\left( s\right) \right) ds & t\in \left( s_{i},t_{i+1}\right] ,\text{ }i=1,...,m%
\end{array}
\right. 
\end{equation*}
where $\mathbf{K}_{\alpha }\left( t\right) =t^{\gamma -1}\mathbf{G}_{\alpha }\left( t\right)$, $\mathbf{G}_{\alpha }\left( t\right) =\displaystyle\int_{0}^{\infty }\alpha \theta \mathbf{M}_{\alpha }\left( \theta \right) \mathbb{P}_{\alpha ,\beta }\left( t^{\alpha }\theta \right) d\theta $ and $\mathbb{P}_{\alpha ,\beta }\left( t\right) =I_{\theta }^{\beta \left( 1-\alpha \right) }\mathbf{K}_{\alpha }\left( t\right)$.
\end{definition}

%%%%%%%%%%%%%%%%%%%%%%%%%%%%%%%%%%%%%%%%%%%%%%%%%%%%%%%%%%%%%%%%%%%%%%%%%%%%%%%%%%%%%%%%%%%%%%%%%%%%%%%%%%%%%%%%%%%%%%%
%%%%%%%%%%%%%%%%%%%%%%%%%%%%%%%%%%%%%%%%%%%%%%%%%%%%%%%%%%%%%%%%%%%%%%%%%%%%%%%%%%%%%%%%%%%%%%%%%%%%%%%%%%%%%%%%%%%%%%%
%%%%%%%%%%%%%%%%%%%%%%%%%%%%%%%%%%%%%%%%%%%%%%%%%%%%%%%%%%%%%%%%%%%%%%%%%%%%%%%%%%%%%%%%%%%%%%%%%%%%%%%%%%%%%%%%%%%%%%%
\section{Existence continuous mild solutions}

In this section, we investigate the main result of this paper, the existence of continuous mild solutions for Eq.(\ref{eq1}) using the idea equicontinuity of $C_{0}$-semigroup $(\mathbb{P}_{\alpha}(t))_{t\geq 0}$ and the Lebesgue dominated convergence theorem. However, to achieve such a result, we assume certain conditions:

(A1) The nonlinear function $f:J\times \Omega \rightarrow \Omega $ is continuous, for some $r>0$, there exist a constant $\rho >0$, Lebesgue integrable function, $\varphi :J\rightarrow \left[ 0,\infty \right) $ and a
nondecreasing continuous function $\psi :\left[ 0,\infty \right) \rightarrow \left[ 0,\infty \right) $, such that for all $t\in J$ and $u\in \Omega $ satisfying $\left\Vert u\right\Vert _{C_{1-\gamma }}\leq r$,
\begin{equation*}
\left\Vert f\left( t,u\right) \right\Vert _{C_{1-\gamma }}\leq \varphi \left( t\right) \psi \left( \left\Vert u\right\Vert _{C_{1-\gamma }}\right) \text{ and }\underset{r\rightarrow \infty }{\lim }\inf \frac{\psi \left(
r\right) }{r}=\rho <\infty ;
\end{equation*}

(A2) The impulsive function $\zeta_{k}:\left[ t_{k},s_{k}\right] \times \Omega \rightarrow \Omega $ is continuous and there exists a constant $\mathbf{K}_{\zeta_{k}}>0,$ $k=1,2,...,m$, such that for all $u,v\in \Omega $
\begin{equation*}
\left\Vert \zeta_{k}\left( t,u\right) -\zeta_{k}\left( t,v\right) \right\Vert _{C_{1-\gamma }}\leq \mathbf{K}_{\zeta_{k}}\left\Vert u-v\right\Vert _{C_{1-\gamma }},\text{ }\forall t\in \left( t_{k},s_{k}\right] .
\end{equation*}

(A3) There exists positive constants, $L_{k}$ $\left( k=0,1,...,m\right) $ such that for any countable set $\mathfrak{D}\subset \Omega$,
\begin{equation*}
\mu \left( f\left( t,\mathfrak{D}\right) \right) \leq \mathbf{L}_{k}\mu \left( \mathfrak{D}\right) ,\text{ }t\in \left( s_{k},t_{k+1}\right] ,\text{ }k=0,1,...,m.
\end{equation*}

For brevity of notation, we denote
\begin{equation*}
\mathbf{K}:=\underset{k=1,...,m}{\max }\mathbf{K}_{\zeta_{k}};
\end{equation*}
\begin{equation}\label{eq6}
\Lambda :=\underset{k=0,...,m}{\max }\left\Vert \varphi \right\Vert _{L\left[ s_{k},t_{k+1}\right] }; 
\end{equation}\label{eq7}
\begin{equation}
\mathbf{L}:=\underset{k=0,...,m}{\max }\mathbf{L}_{k}\left( t_{k+1}-s_{k}\right).
\end{equation}

\begin{teorema} Suppose that the semigroup $\left( \mathbb{P}_{\alpha }\left( t\right) \right) _{t\geq 0}$ generated by $-\mathcal{A}$ is equicontinuous, the function $\zeta_{k}\left( \cdot ,\theta \right) $ is bounded for $k=1,2,...,m.$ If the conditions {\rm (A1)-(A3)} are satisfied, then {\rm Eq.(\ref{eq1})} has at least one $\mathbf{PC}_{1-\gamma }-$mild solution $u\in \mathbf{PC}_{1-\gamma }\left( J,\Omega \right)$ provided that
\begin{equation}\label{eq8}
\mathbf{M}\max \left\{ \rho \Lambda +\mathbf{K},\mathbf{K}+4\mathbf{KL}\right\} <1.
\end{equation}
\end{teorema}

\begin{proof} First, we define the following operator $\mathfrak{F}$ on $\mathbf{PC}_{1-\gamma }\left(J,\Omega \right) $ given by
\begin{equation}\label{eq9}
\left( \mathfrak{F}u\right) \left( t\right) =\left( \mathfrak{F}_{1}u\right) \left( t\right) +\left( \mathfrak{F}_{2}u\right) \left( t\right) 
\end{equation}
where
\begin{equation}\label{eq10}
\left( \mathfrak{F}_{1}u\right) \left( t\right) =\left\{ 
\begin{array}{ll}
\mathbb{P}_{\alpha ,\beta }\left( t\right) u_{0}, & t\in \left[ 0,t_{1}\right]  \\ 
\zeta_{k}\left( t,u\left( t\right) \right),  & t\in \left( t_{k},s_{k} \right] ,\text{ }k=1,2,...,m \\ 
\mathbb{P}_{\alpha ,\beta }\left( t\right) \zeta_{k}\left( s_{k},t_{k+1}\right),  & t\in \left( s_{k},t_{k+1}\right] ,\text{ }k=0,1,...,m
\end{array}%
\right.
\end{equation}%
and
\begin{equation}\label{eq11}
\left( \mathfrak{F}_{2}u\right) \left( t\right) =\left\{ 
\begin{array}{ll}
\displaystyle\int_{0}^{t}\mathbf{K}_{\alpha }\left( t-s\right) f\left( s,u\left( s\right) \right) ds, & t\in \left[ 0,t_{1}\right]  \\ 
0, & t\in \left( t_{k},s_{k}\right] ,\text{ }k=1,2,...,m \\ 
\displaystyle\int_{s_{k}}^{t}\mathbf{K}_{\alpha }\left( t-s\right) f\left( s,u\left( s\right) \right) ds, & t\in \left( s_{k},t_{k+1}\right] ,\text{ }k=0,1,...,m.
\end{array}
\right. 
\end{equation}

Note that $\mathfrak{F}$ is well defined and that $\mathbf{PC}_{1-\gamma}-$mild solution of Eq.(\ref{eq1}) is equivalent to the fixed point of operator $\mathfrak{F}$ defined by Eq.(\ref{eq9}). Now, the objective is to prove that the operator $\mathfrak{F}$ admits at least one fixed point. The proof will be carried out in four steps.

\textbf{Step I:} $\mathfrak{F}u\in \mathbf{PC}_{1-\gamma }\left( J,\Omega \right)$, $\forall u\in \mathbf{PC}_{1-\gamma }\left( J,\Omega \right) $. 

Suppose that $0\leq \tau \leq t\leq t_{1}$, then by means of the strongly continuity of the semigroup $\mathbb{P}_{\alpha ,\beta }\left( t\right) \left( t\geq 0\right) $ Eq.(\ref{eq2}) and Eq.(\ref{eq9}), we get
\begin{eqnarray*}
&&\left\Vert t^{1-\gamma }\left[ \left( \mathfrak{F}u\right) \left( t\right) -\left( \mathfrak{F}u\right) \left( \tau \right) \right] \right\Vert   \notag \\
&\leq &\left\Vert t^{1-\gamma }\left[ \mathbb{P}_{\alpha ,\beta }\left( t\right)
u_{0}-\mathbb{P}_{\alpha ,\beta }\left( \tau \right) u_{0}\right] \right\Vert   \notag
\\
&&+\left\Vert t^{1-\gamma }\left[ \int_{0}^{t}\mathbf{K}_{\alpha }\left( t-s\right)f\left( s,u\left( s\right) \right) ds-\int_{0}^{\tau }\mathbf{K}_{\alpha }\left( \tau -s\right) f\left( s,u\left( s\right) \right) ds\right] \right\Vert   \notag
\\
&=&\left\Vert t^{1-\gamma }\left[ \mathbb{P}_{\alpha ,\beta }\left( t\right)u_{0}-\mathbb{P}_{\alpha ,\beta }\left( \tau \right) u_{0}\right] \right\Vert   \notag
\\
&&+\left\Vert t^{1-\gamma }\left[ \int_{\tau }^{t}\mathbf{K}_{\alpha }\left(t-s\right) f\left( s,u\left( s\right) \right) ds+\int_{0}^{\tau }\left(\mathbf{K}_{\alpha }\left( t-s\right) -\mathbf{K}_{\alpha }\left( \tau -s\right) \right)
f\left( s,u\left( s\right) \right) ds\right] \right\Vert   \notag \\
&\leq &\left\Vert t^{1-\gamma }\left[ \mathbb{P}_{\alpha ,\beta }\left( t\right)u_{0}-\mathbb{P}_{\alpha ,\beta }\left( \tau \right) u_{0}\right] \right\Vert +\underset{t\in J}{\sup }\left\Vert \mathbf{K}_{\alpha }\left( t-s\right) \right\Vert
\left\Vert t^{1-\gamma }\int_{\tau }^{t}f\left( s,u\left( s\right) \right)ds\right\Vert   \notag \\
&&+\left\Vert t^{1-\gamma }\int_{0}^{\tau }\left( \mathbf{K}_{\alpha }\left(t-s\right) -\mathbf{K}_{\alpha }\left( \tau -s\right) \right) f\left( s,u\left(s\right) \right) ds\right\Vert   \notag \\
&\leq &\mathbf{M}\left\Vert t^{1-\gamma }\mathbb{P}_{\alpha ,\beta }\left( t-s\right)
u_{0}\right\Vert +\mathbf{M}\left\Vert t^{1-\gamma }\int_{\tau }^{t}f\left( s,u\left(
s\right) \right) ds\right\Vert   \notag \\
&&+\left\Vert t^{1-\gamma }\int_{0}^{\tau }\left( \mathbf{K}_{\alpha }\left( t-\tau \right) \mathbf{K}_{\alpha }\left( \tau -s\right) f\left( s,u\left( s\right) \right) -\mathbf{K}_{\alpha }\left( t-\tau \right) \right) f\left( s,u\left( s\right) \right) ds\right\Vert \left. \rightarrow 0\right. ,
\end{eqnarray*}
as $t\rightarrow \tau$.

In this sense, it follows that $\mathfrak{F}u\in C_{1-\gamma }\left( \left[ 0,t_{1}\right] ,\Omega \right)$. Now it is necessary to check for the other intervals, i.e., $\mathfrak{F}u\in C_{1-\gamma }\left( \left( t_{k},s_{k}\right] ,\Omega \right)$ and $\mathfrak{F}u\in C_{1-\gamma }\left( \left( s_{k},t_{k+1}\right] ,\Omega \right)$, for every $k=1,2,...,m$.

Note that, by means of Eq.(\ref{eq8}) and the continuity of the non-instantaneous impulsive functions $\zeta_{k}\left( t,u\left( t\right) \right)$ with $k=1,2,...,m$, it is easy to know that $\mathfrak{F}u\in C_{1-\gamma }\left( \left( t_{k},s_{k}\right] ,\Omega \right) $, is in fact similar to the proof of continuity carried out above $\left( \mathfrak{F}u\right) \left( t\right) $, for $t\in\left[ 0,t_{1}\right] $, we can prove that $\mathfrak{F}u\in C_{1-\gamma }\left( \left( s_{k},t_{k+1}\right] ,\Omega \right) $, for $\ k=1,2,...,m$. Thus, we conclude that $\mathfrak{F}u\in \mathbf{PC}_{1-\gamma }\left( J,\Omega \right) $ for $u\in \mathbf{PC}_{1-\gamma }\left( J,\Omega \right) $, i.e., $\mathfrak{F}:\mathbf{PC}_{1-\gamma }\left( J,\Omega \right) \rightarrow \mathbf{PC}_{1-\gamma }\left( J,\Omega \right)$. 

\textbf{Step II:} $\exists \mathbf{R}>0$; $\mathfrak{F}\left( \widetilde{\Omega }_{\mathbf{R}}\right) \subset \widetilde{\Omega }_{\mathbf{R}}$.

For this step, will be carried out by contradiction. Suppose that not true, i.e., there is no $\mathbf{R}>0$ such that $\mathfrak{F}\left( \widetilde{\Omega }_{\mathbf{R}}\right) \subset \widetilde{\Omega }_{\mathbf{R}}$, then  in this sense for each $r>0$, $\exists u_{r}\in \Delta _{r}$ and $t_{r}\in J,$ such that $\left\Vert \mathfrak{F}u_{r}\right\Vert _{C_{1-\gamma }}>r$. Now we need to evaluate $t_{r}$ in the intervals $\left[ 0,t_{1}\right]$, $\left( t_{k},s_{k}\right]$ and $\left( s_{k},t_{k+1}\right]$. Thus, we have:

If $t_{r}\in \left[ 0,t_{1}\right]$, then by Eq.(\ref{eq2}) and Eq.(\ref{eq9}) and (A1), we obtain
\begin{eqnarray}\label{eq12}
\left\Vert t^{1-\gamma }\left[ \left( \mathfrak{F}u_{r}\right) \left( t_{r}\right)\right] \right\Vert  &=&\left\Vert t^{1-\gamma }\left[ \mathbb{P}_{\alpha ,\beta }\left( t_{r}\right) u_{0}+\int_{0}^{t_{r}}\mathbf{K}_{\alpha }\left( t_{r}-s\right)
f\left( s,u_{r}\left( s\right) \right) ds\right] \right\Vert   \notag \\
&\leq &\left\Vert t^{1-\gamma }\mathbb{P}_{\alpha ,\beta }\left( t_{r}\right) u_{0}\right\Vert +\left\Vert t^{1-\gamma }\int_{0}^{t_{r}}\mathbf{K}_{\alpha }\left( t_{r}-s\right) f\left( s,u_{r}\left( s\right) \right) ds\right\Vert   \notag \\
&\leq &\mathbf{M}\left\Vert t^{1-\gamma }u_{0}\right\Vert +\mathbf{M}\left\Vert t^{1-\gamma }\right\Vert \left\Vert \int_{0}^{t_{r}}s^{\gamma -1}\varphi \left( s\right) \Psi \left( \left\Vert u\right\Vert _{C_{1-\gamma }}\right) ds\right\Vert  \notag \\ &\leq &\mathbf{M}\left\Vert u_{0}\right\Vert _{C_{1-\gamma }}+\mathbf{M}\Psi \left( r\right) \left\Vert \varphi \right\Vert _{\mathcal{L}\left[ 0,t_{1}\right] }.
\end{eqnarray}

If $t_{r}\in \left( t_{k},s_{k}\right]$, $k=1,2,...,m,$ then by Eq.(\ref{eq2}), Eq.(\ref{eq9}) and (A2), we get
\begin{eqnarray}\label{eq13}
\left\Vert t^{1-\gamma }\left[ \left( \mathfrak{\mathfrak{F}}u_{r}\right) \left( t_{r}\right) \right] \right\Vert  &\leq &\mathbf{K}_{\zeta_{k}}\left\Vert u_{r}\left( t_{r}\right) \right\Vert _{C_{1-\gamma }}  \notag \\
&\leq &\mathbf{K}_{\zeta_{k}}\left\Vert u_{r}\left( t_{r}\right) \right\Vert _{C_{1-\gamma }}+\left\Vert \zeta_{k}\left( t_{r},\theta \right) \right\Vert _{C_{1-\gamma }}  \notag \\
&\leq &\mathbf{K}_{\zeta_{k}}+\mathbf{N} 
\end{eqnarray}
where 
\begin{equation*}
\mathbf{N}=\underset{k=1,2,...,m}{\max }\underset{t\in J}{\sup }\left\Vert \zeta_{k}\left( t_{r},\theta \right) \right\Vert _{C_{1-\gamma }}.
\end{equation*}

On the other hand, if $t_{r}\in \left( s_{k},t_{k+1}\right]$, $k=1,2,...,m,$ then by Eq(\ref{eq2}), Eq.(\ref{eq9}) and (A1) and (A2), we have
\begin{eqnarray}\label{eq14}
&&\left\Vert t^{1-\gamma }\left[ \left( \mathfrak{F}u_{r}\right) \left( t_{r}\right)\right] \right\Vert  \notag\\
&=&\left\Vert t^{1-\gamma }\mathbb{P}_{\alpha ,\beta }\left( t_{r}\right) \zeta_{k}\left( s_{k},u\left( s_{k}\right) \right)
+t^{1-\gamma }\int_{s_{k}}^{t_{r}}\mathbf{K}_{\alpha }\left( t_{r}-s\right) f\left( s,u_{r}\left( s\right) \right) ds\right\Vert   \notag \\
&\leq &\mathbf{M}\left( \mathbf{K}_{\zeta_{k}}\left\Vert u_{r}\left( s_{r}\right) \right\Vert _{C_{1-\gamma }}+\left\Vert \zeta_{k}\left( t_{r},\theta \right) \right\Vert _{C_{1-\gamma }}\right) +\left\Vert t^{1-\gamma
}\int_{s_{k}}^{t_{r}}\mathbf{K}_{\alpha }\left( t_{r}-s\right) f\left( s,u_{r}\left( s_{r}\right) \right) ds\right\Vert   \notag \\ 
&\leq &\mathbf{M}\left( \mathbf{K}_{\zeta_{k}}+\mathbf{N}\right) +\mathbf{M}t^{1-\gamma }\int_{s_{k}}^{t_{r}}s^{\gamma -1}\left\Vert s^{1-\gamma }f\left( s,u_{r}\left( s_{r}\right) \right) ds\right\Vert   \notag \\ 
&\leq &\mathbf{M}\left( \mathbf{K}_{\zeta_{k}}+\mathbf{N}\right) +\mathbf{M}\Psi \left( r\right) \left\Vert \varphi \right\Vert _{\mathcal{L}\left[ s_{k},t_{k+1}\right] }.
\end{eqnarray}

By means of Eq.(\ref{eq2}), Eq.(\ref{eq6}), Eq.(\ref{eq9}), Eq.(\ref{eq12})-Eq.(\ref{eq14}) and using the fact $r<\left\Vert \mathfrak{F}u_{r}\right\Vert$, we have
\begin{equation}\label{eq15}
r<\left\Vert t^{1-\gamma }\left( \mathfrak{F}u_{r}\right) \left( t_{r}\right) \right\Vert \leq \mathbf{M}\left( \left\Vert u_{0}\right\Vert _{C_{1-\gamma }}+\Psi \left( r\right) \Lambda +\mathbf{K}_{r}+\mathbf{N}\right) 
\end{equation}
where $\Lambda =\underset{k=0,1,...,m}{\max }\left\Vert \varphi \right\Vert_{L\left[ s_{k},t_{k+1}\right] }.$

Finally, we divide both sides of Eq.(\ref{eq15}) by $r$ and taking the limit as $r\rightarrow \infty $, we obtain
\begin{eqnarray*}
1 &<&\frac{\left\Vert t^{1-\gamma }\left( \mathfrak{F}u_{r}\right) \left( t_{r}\right)
\right\Vert }{r} \leq \frac{\mathbf{M}}{r}\left( \left\Vert u_{0}\right\Vert _{C_{1-\gamma }}+\Psi
\left( r\right) \Lambda +\mathbf{K}_{r}+\mathbf{N}\right) =\mathbf{M}\left( \rho \Lambda +\mathbf{K}\right).
\end{eqnarray*}

Note that, is a contradiction the Eq.(\ref{eq8}). Therefore, we conclude the proof.

\textbf{Step III:} $\mathfrak{F}_{1}:\widetilde{\Omega }_{\mathbf{R}}\rightarrow \widetilde{\Omega }_{\mathbf{R}}$ and $\mathfrak{F}_{2}:\widetilde{\Omega }_{\mathbf{R}}\rightarrow \widetilde{\Omega }_{\mathbf{R}}$ are Lipschitz continuous.

Then, for $t\in \left( s_{k},t_{k+1}\right]$, $k=1,2,...,m$ and $u,v\in \widetilde{\Omega }_{\mathbf{R}}$ by Eq.(\ref{eq10}) and (A2), we obtain
\begin{eqnarray}\label{eq16}
\left\Vert t^{1-\gamma }\left[ \left( \mathfrak{F}_{1}u\right) \left( t\right) -\left( \mathfrak{F}_{1}v\right) \left( t\right) \right] \right\Vert  &=&\left\Vert t^{1-\gamma }\mathbb{P}_{\alpha ,\beta }\left( t\right) \left[ \zeta_{k}\left( s_{k},u\left( s_{k}\right) \right) -\zeta_{k}\left( s_{k},v\left( s_{k}\right) \right) \right] \right\Vert   \notag \\ 
&\leq &\mathbf{M}\left\Vert t^{1-\gamma }\left[ \zeta_{k}\left( s_{k},u\left(s_{k}\right) \right) -\zeta_{k}\left( s_{k},v\left( s_{k}\right) \right) \right] \right\Vert   \notag \\
&\leq &\mathbf{M}\mathbf{K}_{\zeta_{k}}\left\Vert u-v\right\Vert _{\mathbf{PC}_{1-\gamma }.\text{ }}
\end{eqnarray}

On the other hand, for $t\in \left( s_{k},t_{k}\right]$, $k=1,2,...,m$ and $u,v\in \widetilde{\Omega }_{\mathbf{R}}$, by Eq.(\ref{eq10}) and (A2), we get
\begin{eqnarray}\label{eq17}
\left\Vert t^{1-\gamma }\left[ \left( \mathfrak{F}_{1}u\right) \left( t\right) -\left( \mathfrak{F}_{1}v\right) \left( t\right) \right] \right\Vert  &\leq &\mathbf{K}_{\zeta_{k}}\left\Vert \left[ \zeta_{k}\left( t,u\right) -\zeta_{k}\left( t,v\right) \right] \right\Vert _{C_{1-\gamma }}  \notag \\
&\leq &\mathbf{K}_{\zeta_{k}}\left\Vert u-v\right\Vert _{\mathbf{PC}_{1-\gamma }\text{ }}.
\end{eqnarray}

By means of Eq.(\ref{eq16}), Eq.(\ref{eq17}), Eq.(\ref{eq2}) and Eq.(\ref{eq6}), we have
\begin{eqnarray}\label{eq18}
\left\Vert \mathfrak{F}_{1}u-\mathfrak{F}_{1}v\right\Vert _{\mathbf{PC}_{1-\gamma }} &\leq &\mathbf{M}\mathbf{K}_{\zeta_{k}}\left\Vert u-v\right\Vert _{\mathbf{PC}_{1-\gamma }}+\mathbf{K}_{\zeta_{k}}\left\Vert u-v\right\Vert _{\mathbf{PC}_{1-\gamma }}  \notag \\
&\leq &\mathbf{M}\underset{k=1,2,...,m}{\max }\mathbf{K}_{\zeta_{k}}\left\Vert u-v\right\Vert_{\mathbf{PC}_{1-\gamma }}  \notag \\
&=&\mathbf{M}\mathbf{K}\left\Vert u-v\right\Vert _{\mathbf{PC}_{1-\gamma }}.
\end{eqnarray}

Therefore $\mathfrak{F}_{1}$ is continuous in $\widetilde{\Omega }_{\mathbf{R}}$. Now, we prove that $\mathfrak{F}_{2}$ is continuous in $\widetilde{\Omega }_{\mathbf{R}}$. Consider $u_{n}\in \widetilde{\Omega }_{\mathbf{R}}$ be a sequence such that $\underset{n\rightarrow \infty }{\lim }u_{n}=u$ in $\widetilde{ \Omega }_{\mathbf{R}}.$ By the continuity of nonlinear term $f$ with respect to the second variable, for $s\in J$, we get
\begin{equation}\label{eq19}
\underset{n\rightarrow \infty }{\lim }f\left( s,u_{n}\left( s\right) \right) =f\left( s,u\left( s\right) \right).
\end{equation}

Using (A1), we have
\begin{eqnarray}\label{eq20}
\left\Vert t^{1-\gamma }\left[ f\left( s,u_{n}\left( s\right) \right) -f\left( s,u\left( s\right) \right) \right] \right\Vert  &\leq &\left\Vert f\left( s,u_{n}\right) \right\Vert _{C_{1-\gamma }}+\left\Vert f\left(
s,u\right) \right\Vert _{C_{1-\gamma }}  \notag \\
&\leq &\varphi \left( s\right) \Psi \left( \left\Vert u_{n}\right\Vert_{C_{1-\gamma }}\right) +\varphi \left( s\right) \Psi \left( \left\Vert u\right\Vert _{C_{1-\gamma }}\right)   \notag \\
&=&2\varphi \left( s\right) \Psi \left( \mathbf{R}\right),
\end{eqnarray}
for $s\in J$.

As $s\rightarrow 2\varphi \left( s\right) \Psi \left( \mathbf{R}\right) $ is Lebesgue integrable for $s\in \left[ s_{k},t\right] $ and $t\in \left( s_{k},t_{k+1}\right] ,k=0,1,...,m$, then by means of the Eq.(\ref{eq2}), Eq.(\ref{eq11}), Eq.(\ref{eq19}), Eq.(\ref{eq20}) and the Lebesgue dominated convergence theorem, we obtain
\begin{eqnarray}\label{eq21}
\left\Vert t^{1-\gamma }\left[ \left( \mathfrak{F}_{2}u_{n}\right) \left( t\right) -\left( \mathfrak{F}_{2}u\right) \left( t\right) \right] \right\Vert  &\leq &t^{1-\gamma }\int_{s_{k}}^{t}\mathbf{K}_{\alpha }\left( t-s\right) s^{\gamma
-1}\left\Vert s^{1-\gamma }\left[ f\left( s,u_{n}\left( s\right) \right) -f\left( s,u\left( s\right) \right) \right] \right\Vert ds  \notag \\
&\leq &\mathbf{M}\int_{s_{k}}^{t}\left\Vert s^{1-\gamma }\left[ f\left( s,u_{n}\left( s\right) \right) -f\left( s,u\left( s\right) \right) \right] \right\Vert ds\rightarrow 0
\end{eqnarray}
as $n\rightarrow \infty $. 

Then, $\left\Vert \mathfrak{F}_{2}u_{n}-\mathfrak{F}_{2}u\right\Vert _{\mathbf{PC}_{1-\gamma }}\rightarrow 0$ as $n\rightarrow \infty $ which means that $\mathfrak{F}_{2}$ defined by Eq.(\ref{eq11}) is continuous in $\widetilde{\Omega }_{\mathbf{R}}$.

\textbf{Step IV:} $\mathfrak{F}_{2}:\widetilde{\Omega }_{\mathbf{R}}\rightarrow \widetilde{\Omega }_{\mathbf{R}}$ is equicontinuous. 

Then, for any $u\in \widetilde{\Omega }_{\mathbf{R}}$ and $s_{k}<t_{1}<t_{2}\leq t_{k+1}$ for $k=0,1,...,m
$, we have
\begin{eqnarray*}
\left\Vert t^{1-\gamma }\left[ \left(\mathfrak{F}_{2}u\right) \left( t_{2}\right) -\left( \mathfrak{F}_{2}u\right) \left( t_{1}\right) \right] \right\Vert  &\leq &\left\Vert \int_{t_{1}}^{t_{2}}\mathbf{K}_{\alpha }\left( t_{2}-s\right) f\left( s,u\left( s\right) \right) ds\right\Vert _{C_{1-\gamma }}\notag \\&&+\left\Vert\int_{s_{k}}^{t_{1}}\left( \mathbf{K}_{\alpha }\left( t_{2}-s\right) -\mathbf{K}_{\alpha }\left( t_{1}-s\right) \right) f\left( s,u\left( s\right) \right) ds\right\Vert _{C_{1-\gamma }}  \notag \\
&&:=I_{1}+I_{2}
\end{eqnarray*}
where $I_{1}:=\left\Vert \displaystyle\int_{t_{1}}^{t_{2}}\mathbf{K}_{\alpha }\left( t_{2}-s\right) f\left( s,u\left( s\right) \right) ds\right\Vert _{C_{1-\gamma }}$ and \\$I_{2}:=\left\Vert \displaystyle\int_{s_{k}}^{t_{1}}\left(\mathbf{K}_{\alpha }\left( t_{2}-s\right) -\mathbf{K}_{\alpha }\left( t_{1}-s\right) \right) f\left( s,u\left( s\right) \right) ds\right\Vert _{C_{1-\gamma }}$.

Now, we check $I_{1}$ and $I_{2}$ tend to $0$ independently of $u\in \widetilde{\Omega }_{\mathbf{R}}$ when $t_{2}-t_{1}\rightarrow 0$. By means of Eq.(\ref{eq2}) and (A1), we get
\begin{eqnarray*}
I_{1} &\leq &\int_{t_{1}}^{t_{2}}\left\Vert \mathbf{K}_{\alpha }\left( t_{2}-s\right) \right\Vert _{C_{1-\gamma }}\left\Vert f\left( s,u\left( s\right) \right) \right\Vert _{C_{1-\gamma }}ds  \notag \\
&\leq &\mathbf{M}\int_{t_{1}}^{t_{2}}\varphi \left( s\right) \Psi \left( \left\Vert u\right\Vert _{C_{1-\gamma }}\right) ds  \notag \\
&\leq &\mathbf{M}\Psi \left( \mathbf{R}\right) \int_{t_{1}}^{t_{2}}\varphi \left( s\right) ds\rightarrow 0
\end{eqnarray*}
as $t_{2}-t_{1}\rightarrow 0$.

Now, for $\varepsilon >0$ small enough and by means of the Eq.(\ref{eq2}), (A1), equicontinuity of the $C_{0}-$semigroup $\left( \mathbb{P}_{\alpha }\left( t\right) \right) _{t\geq 0}$ and the Lebesgue dominated convergence theorem, we get
\begin{eqnarray*}
I_{2} &\leq &\left\Vert \int_{t_{1}}^{t_{1}-\varepsilon }\left( \mathbf{K}_{\alpha }\left( t_{2}-s\right) -\mathbf{K}_{\alpha }\left( t_{1}-s\right) \right) ds\right\Vert_{C_{1-\gamma }} +\left\Vert \int_{t_{1}-\varepsilon }^{t_{1}}\left( \mathbf{K}_{\alpha
}\left( t_{2}-s\right) -\mathbf{K}_{\alpha }\left( t_{1}-s\right) \right) ds\right\Vert _{C_{1-\gamma }}  \notag \\
&\leq &\Psi \left( \mathbf{R}\right) \int_{s_{k}}^{t_{1}-\varepsilon }\left\Vert \mathbf{K}_{\alpha }\left( t_{2}-s\right) -\mathbf{K}_{\alpha }\left( t_{1}-s\right) \right\Vert _{C_{1-\gamma }}\varphi \left( s\right) ds+2\mathbf{M}\Psi \left(
\mathbf{R}\right) \int_{t_{1}-\varepsilon }^{t_{1}}\varphi \left( s\right) ds  \notag \\
&=&\Psi \left( \mathbf{R}\right) \int_{t\varepsilon }^{t_{1}-\varepsilon }\left\Vert \mathbf{K}_{\alpha }\left( t_{2}-t_{1}-s\right) -\mathbf{K}_{\alpha }\left( s\right) \right\Vert _{C_{1-\gamma }}\varphi \left( t_{1}-s\right) ds+2\mathbf{M}\Psi \left( \mathbf{R}\right) \int_{t_{1}-\varepsilon }^{t_{1}}\varphi \left( s\right) ds\rightarrow 0
\end{eqnarray*}
as $t_{2}-t_{1}\rightarrow 0$ and $\varepsilon \rightarrow 0$.

Note that the result, $\left\Vert t^{1-\gamma }\left[ \left( \mathfrak{F}_{2}u\right) \left( t_{2}\right) -\left( \mathfrak{F}_{2}u\right) \left( t_{1}\right) \right] \right\Vert _{C_{1-\gamma }} \rightarrow 0$ independently of $u\in \widetilde{\Omega }_{\mathbf{R}} $ as $t_{2}-t_{1}\rightarrow 0$, which means that $\mathfrak{F}_{2}:\widetilde{\Omega }_{\mathbf{R}}\rightarrow \widetilde{\Omega }_{\mathbf{R}}$ is equicontinuous.

For any bounded $\mathfrak{D}\subset \widetilde{\Omega }_{\mathbf{R}},$ by Lemma \ref{lem3}, we know that there exists a countable set $\mathfrak{D}_{0}=\left\{ u_{n}\right\} \subset \mathfrak{D},$ such that
\begin{equation}\label{eq22}
\mu \left( \mathfrak{F}_{2}\left( \mathfrak{D}\right) \right) _{\mathbf{PC}{1-\gamma }}\leq 2\mu \left( \mathfrak{F}_{2}\left( \mathfrak{D}_{0}\right) \right) _{\mathbf{PC}_{1-\gamma }}.
\end{equation}

Since $\mathfrak{F}_{2}\left( \mathfrak{D}_{0}\right) \subset \mathfrak{F}_{2}\left( \widetilde{\Omega }_{\mathbf{R}}\right) $ is bounded and equicontinuous, by means of the Lemma \ref{lem5}, we know that 
\begin{equation}\label{eq23}
\mu \left( \mathfrak{F}_{2}\left( \mathfrak{D}_{0}\right) \right) _{\mathbf{PC}_{1-\gamma }}=\underset{\underset{k=0,1,...,m}{t\in \left[ s_{k},t_{k+1}\right] }}{\max }\mu \left( \mathfrak{F}_{2}\left( \mathfrak{D}_{0}\right) \right) \left( t\right) .
\end{equation}

On the other hand, for $t\in \left[ s_{k},t_{k+1}\right]$ with $k=0,1,...,m$, using the Lemma \ref{lem4}, (A3) and Eq.(\ref{eq11}), we have
\begin{eqnarray}\label{eq24}
\mu \left( \mathfrak{F}_{2}\left( \mathfrak{D}_{0}\right) \right) \left( t\right)  &\leq &\mu \left( \left\{ \mathbf{M}\int_{s_{k}}^{t}f\left( s,u_{n}\left( s\right) \right) \right\} ds\right)   \notag \\
&\leq &2\mathbf{M}\int_{s_{k}}^{t}L_{k}\mu \left( u_{n}\left( s\right) \right) ds \notag \\
&\leq &2\mathbf{M}\mathbf{L}_{k}\mu \left( \mathfrak{D}\right) \int_{s_{k}}^{t}ds  \notag \\
&\leq &2\mathbf{M}\mathbf{L}_{k}\mu \left( \mathfrak{D}\right) _{\mathbf{PC}_{1-\gamma }}\left(t_{k+1}-s_{k}\right) .
\end{eqnarray}

Therefore, by means of the Eq.(\ref{eq8}), Eq.(\ref{eq22}), Eq.(\ref{eq23}) and Eq.(\ref{eq24}), we obtain
\begin{eqnarray}\label{eq25}
\mu \left( \mathfrak{F}_{2}\left( \mathfrak{D}_{0}\right) \left( t\right) \right)  &\leq &2\mathbf{M}\mathbf{L}_{k}\mu \left( \mathfrak{D}\right) _{\mathbf{PC}_{1-\gamma }}\left( t_{k+1}-s_{k}\right)  \notag \\
&\leq &4\mathbf{M}\mathbf{L}\mu \left( \mathfrak{D}\right) _{\mathbf{PC}_{1-\gamma }}.
\end{eqnarray}

From Eq.(\ref{eq18}) and Lemma \ref{lem1} (7), for any bounded $\mathfrak{D}\subset  \widetilde{\Omega }_{\mathbf{R}}$,  we know that is true
\begin{equation}\label{eq26}
\mu \left( \mathfrak{F}_{1}\left( \mathfrak{D}\right) \right) _{\mathbf{PC}_{1-\gamma }}\leq \mathbf{M}\mathbf{K}\mu \left(\mathfrak{D}\right) _{\mathbf{PC}_{1-\gamma }}.
\end{equation}

Using the Eq.(\ref{eq25}), Eq.(\ref{eq26}) and Lemma \ref{lem1} (6), we get
\begin{equation}\label{eq27}
\mu \left( \mathfrak{F}\left( \mathfrak{D}\right) \right) _{\mathbf{PC}_{1-\gamma }}\leq \mu \left(\mathfrak{F}_{1}\left( \mathfrak{D}\right) \right) _{\mathbf{PC}_{1-\gamma }}+\mu \left( \mathfrak{F}_{2}\left(\mathfrak{D}\right) \right) _{\mathbf{PC}_{1-\gamma }}\leq \mathbf{M}\left( \mathbf{K}+4\mathbf{L}\right) \mu \left(\mathfrak{D}\right) _{\mathbf{PC}_{1-\gamma }}.
\end{equation}

Now, combining the Eq.(\ref{eq27}), Eq.(\ref{eq8}) and Definition \ref{def3}, we have $\mathfrak{F}:\widetilde{\Omega }_{\mathbf{R}}\rightarrow \widetilde{\Omega }_{\mathbf{R}}$ is a $k$-set contractive. Thus, through Lemma \ref{lem2} has at least one fixed point $u\in \widetilde{\Omega }_{\mathbf{R}}$, which is just a $\mathbf{PC}_{1-\gamma }$ mild solution of Eq.(\ref{eq1}).

\end{proof}

%%%%%%%%%%%%%%%%%%%%%%%%%%%%%%%%%%%%%%%%%%%%%%%%%%%%%%%%%%%%%%%%%%%%%%%%%%%%%%%%%%%%%%%%%%%%%%%%%%%%%%%%%%%%%%%%%%%%%%%%%%%%%%%%%%%%%%%%%%%%%%%%%%%%%%%%%%%%%%%%%%%%%%%%%%%%%%%%%%%%%%%%%%%%%%%%
\section{Concluding remarks}

We conclude the paper, with the objective reached, i.e., we investigate the existence of a mild solution for a new class of semi-linear evolution fractional differential equations with non-instantaneous impulses in the sense of Hilfer in the Banach space by means of $C_{0}$-semigroup equicontinuous and the Lebesgue dominated convergence theorem. However, the following question is raised that is open: Will it be possible to investigate the existence and uniqueness of mild solutions of Eq.(\ref{eq1}) or another fractional differential equation in the sense $\psi$-Hilfer fractional derivative? Although Sousa and Oliveira \cite{der3} have recently introduced a version called a Leibniz type rule, the answer is initially no, since there is not yet an integral transform formulation, in particular the Laplace transform, to obtain a mild solution according to Definition 5, since it is necessary and sufficient condition to obtain the expression of the mild solution according to Eq.(\ref{eq1}). Research in this sense, has been developed and consequently contributes a lot to the area.

\section*{Acknowledgment}
{\bf We thank Prof. Dr. E. Capelas de Oliveira for fruitful discussions for suggesting several references on the subject. We also thank have been financially supported by PNPD-CAPES scholarship of the Pos-Graduate Program in Applied Mathematics IMECC-Unicamp.}

\bibliography{ref}
\bibliographystyle{plain}

\end{document}